\documentclass[11pt]{amsart}
\usepackage{amssymb, amstext, amscd, amsmath,color}

\makeatletter
\def\@cite#1#2{{\m@th\upshape\bfseries%
[{#1\if@tempswa{\m@th\upshape\mdseries, #2}\fi}]}}
\makeatother

\theoremstyle{plain}
\newtheorem{thm}{Theorem}[section]
\newtheorem{cor}[thm]{Corollary}
\newtheorem{prop}[thm]{Proposition}
\newtheorem{lem}[thm]{Lemma}

\theoremstyle{definition}
\newtheorem{rem}[thm]{Remark}

\newtheorem{defn}[thm]{Definition}

\newtheorem{eg}[thm]{Example}

\newcommand{\bB}{{\mathbb{B}}}
\newcommand{\bC}{{\mathbb{C}}}

\newcommand{\bF}{{\mathbb{F}}}

  \newcommand{\C}{{\mathcal{C}}}

  \newcommand{\I}{{\mathcal{I}}}
  
  \newcommand{\K}{{\mathcal{K}}}
\renewcommand{\L}{{\mathcal{L}}}
  \newcommand{\M}{{\mathcal{M}}}

  \newcommand{\U}{{\mathcal{U}}}

\renewcommand{\phi}{\varphi}
\newcommand{\upchi}{{\raise.35ex\hbox{\ensuremath{\chi}}}}

\newcommand{\fA}{{\mathfrak{A}}}

\newcommand{\fJ}{{\mathfrak{J}}}

\newcommand{\fL}{{\mathfrak{L}}}

\newcommand{\fR}{{\mathfrak{R}}}

\newcommand{\dist}{\operatorname{dist}}

\newcommand{\Lat}{\operatorname{Lat}}
\newcommand{\CycLat}{\operatorname{CycLat}}
\newcommand{\ran}{\operatorname{Ran}}

\newcommand{\spn}{\operatorname{span}}

\newcommand{\ip}[1]{\langle #1 \rangle}

\newcommand{\norm}[1]{\left\| #1 \right\|}

\newcommand{\ol}{\overline}

\newcommand{\wot}{\textsc{wot}}

\newcommand{\Hinf}{H^\infty }

\newcommand{\inp}[2]{\left \langle #1, #2 \right\rangle}
\begin{document}

\title[A Toeplitz corona theorem for complete NP spaces]%
{The Toeplitz corona problem for algebras of multipliers on a Nevanlinna-Pick space }

\author[R. Hamilton]{Ryan Hamilton}
\address{Pure Math.\ Dept.\\U. Waterloo\\Waterloo, ON\;
N2L--3G1\\CANADA}
\email{rhamilto@math.uwaterloo.ca}

\author[M. Raghupathi]{Mrinal Raghupathi}
\address{1326 Stevenson Center, Department of Mathematics, Vanderbilt University, Nashville, TN, 37240}
\email{mrinal.raghupathi@vanderbilt.edu}

\begin{abstract}
  Suppose $\fA$ is an algebra of operators on a Hilbert space $H$ and
  $A_1, \dots, A_n \in \fA$.  If the row operator $[A_1, \dots, A_n]
  \in B(H^{(n)},H)$ has a right inverse in $B(H, H^{(n)})$, the
  Toeplitz corona problem for $\fA$ asks if a right inverse can be
  found with entries in $\fA$.  When $H$ is a complete Nevanlinna-Pick
  space and $\fA$ is a weakly-closed algebra of multiplication
  operators on $H$, we show that under a stronger hypothesis, the
  corona problem for $\fA$ has a solution.  When $\fA$ is the full
  multiplier algebra of $H$, the Toeplitz corona theorems of Arveson,
  Schubert and Ball-Trent-Vinnikov are obtained.  A tangential
  interpolation result for these algebras is developed in order to
  solve the Toeplitz corona problem.

\end{abstract}

\subjclass[2000]{Primary 47A57; Secondary 30E05, 46E22}
\keywords{NP interpolation, reproducing kernel, corona theorem}
\thanks{First author partially supported by an NSERC graduate scholarship.}
\thanks{}

\date{}
\maketitle

\section{Introduction}\label{S:Introduction}
This paper focuses on the extension of some recently obtained results
in tangential interpolation to a large class of algebras of operators acting on reproducing kernel Hilbert spaces.  
In particular, these results apply to arbitrary weakly-closed algebras of multipliers on
Drury-Arveson space and all other complete Nevanlinna-Pick (NP) spaces.
A consequence of these results is an operator corona
theorem (often called a Toeplitz corona theorem) for these algebras. This work was
motivated by results of the second author and Wick~\cite{RagWick}
which contains a NP theorem ``up to a constant''. The
application of this result to the study of interpolating sequences is
the subject of work currently in progress.

Considerable attention has been given to finding analogues of the
NP theorem, Carleson's corona theorem and Carleson's
interpolation theorem to several variables. In one
variable, the disk is a natural domain for function theory. In higher
dimensions there are considerable differences between the polydisk and
the ball.

In 1978, Drury~\cite{Drury} provided a generalization of von Neumann's
inequality to several variables. In order to do this, Drury introduced
a space of analytic functions on the unit ball $\bB_d$, now called
the Drury-Arveson space, and its associated multiplier
algebra. Popescu~\cite{Popescu} and Arveson~\cite{Arv} carried out a
detailed analysis of the structure of this space. In the 1990s the work of 
McCullough~\cite{McCull92, McCull94}, Quiggin~\cite{Quig93, Quig94}, and Agler-McCarthy~\cite{AMc00} showed that the Drury-Arveson space is the prototype for complete NP spaces.

Our approach is based on reformulating the interpolation problem as a
distance problem and computing a distance formula. The formula is
obtained through the use of Banach space duality. This approach was
pioneered by Sarason~\cite{Sar67} and has been used to tackle
interpolation problems in many contexts~\cite{DP98b,McCull96}. The
first appearance of a Toeplitz corona theorem is in the work of
Arveson~\cite{Arveson}. This work also contained
the distance formula for nest algebras and introduced
the concept of hyper-reflexivity.
 Our main interpolation result also makes use of a distance
formula similar in spirit to the original work of
Arveson. Following this theme, a Toeplitz corona theorem is deduced from 
this formula.

There are other approaches to the tangential interpolation problem and
Toeplitz corona problem. Schubert~\cite{Schubert} approached the
problem through the commutant lifting theorem of Sz.-Nagy and Foias
and also obtained the best bounds for the solution. Commutant lifting
also appears in the (unpublished) work of Helton in this area.
McCullough~\cite{McCull96} used duality techniques to prove a NP
interpolation theorem for dual algebras. McCullough then establishes a
Pick type theorem for uniform algebras that are approximating in
modulus. His work also reproves and unifies several well-known results
in interpolation. Our approach owes a great deal to his insights.

The results in this paper build on the approach to interpolation in
dual algebras made by Davidson and Hamilton~\cite{DH}  and by Raghupathi and Wick~\cite{RagWick} for tangential
interpolation problems in subalgebras of $H^\infty$. 
In keeping with the approach taken in these papers we begin by proving
an abstract theorem about families of kernels and their relation
to a certain distance formula. This is the content of Section~\ref{S:tnp}. 
In Section~\ref{S:predual}, we show that if an algebra of multipliers
has a certain factorization property of weak$^*$-continuous functionals 
called $A_1(1)$, then an exact distance formula is achieved.
Following this, we show in Section ~\ref{S:ds}  that an arbitrary weak$^*$-closed alegbra of multipliers on a complete NP space has this property $A_1(1)$. Consequently, the required distance formula, which implies the tangential interpolation result, is obtained.
Finally, the Toeplitz corona theorem is proved for these algebras in Section~\ref{S:TC}.

\section{Tangential NP families}%
\label{S:tnp}
Suppose $X$ is a set, and $H$ is a Hilbert space of complex-valued functions on $X$.  We call $H$ a reproducing kernel Hilbert space if point evaluations are continuous.
In this case, there is a unique element
of the space $H$ such that $f(x) = \inp{f}{k_x}$ for all $f$ in $H$. The
positive semi-definite function $K(x,y) = \inp{k_y}{k_x}$ is called
the reproducing kernel of $H$. The kernel function uniquely determines
$H$ and there is a well-known correspondence between reproducing
kernel Hilbert spaces and positive semi-definite functions.

For every reproducing kernel Hilbert space $H$, we can associate to it the algebra of multipliers $M(H)$. The algebra $M(H)$ is the set of complex-valued functions on $X$ that pointwise multiply $H$ into itself. By the closed-graph theorem, each multiplier $f\in M(H)$ induces the bounded multiplication operator $M_f\in B(H)$. The key property of a multiplier is that $M_f^*k_x = \overline{f(x)}k_x$ for all $x\in X$ and this property characterizes multipliers.

Every closed subspace $L$ of $H$ is a also a reproducing kernel
Hilbert space and the kernel function of $L$ is given by
$k^L(x_j,x_i) = \langle P_Lk_i, k_j \rangle$.  It could be
the case that for some $x$ in $X$, we have $\langle f, k_x
\rangle = f(x) = 0$ for every $f \in L$, and we allow this
possibility.  We will call a unital weak-$*$-closed subalgebra $\fA$ of
$M(H)$ a \emph{dual algebra of multipliers} and denote its lattice of
invariant subspaces by $\Lat (\fA)$. For $L \in \Lat(\fA)$, every
function $f\in \fA$ defines the multiplication operator $M_f^L$ on
$L$.  Evidently, $M_f^L = M_f|_L$.

The collection of kernels $\{k^L : L \in \Lat (\fA)\}$ is
\emph{realizable} in the sense that
\[\fA = \bigcap_{L \in \Lat (\fA)} M(L). \]
The cyclic invariant subspace for $\fA$ generated by a function $h\in
H$ will be denoted $\fA [h]$ and the set of all cyclic invariant
subspaces for $\fA$ will be denoted $\CycLat (\fA)$.

The column space over $\fA$ will be denoted $C(\fA)$. This is the set
of all operators $T\in B(H, H\otimes \ell^2)$ with entries from
$\fA$. There is a natural identification between $C(\fA)$ and the set
of multipliers between $H$ and $H\otimes \ell^2$.  For $F = [f_1, f_2,
\dots]^T \in C(\fA)$, the associated multiplication operator is given
by $M_F = [M_{f_1},M_{f_2} \dots]^T$.  For $L \in \Lat(A)$, $M_F^L =
[M_{f_1}^L, M_{f_2}^L, \dots]^T$ will denote the multiplier from $L$
into $L \otimes \ell^2$.  The row space $R(\fA)$ is defined
analogously.

The tangential interpolation problem specifies $n$ points
$x_1,\ldots,x_n\in X$, $w_1,\ldots,w_n\in\bC$ and vectors
$v_1,\ldots,v_n\in \ell^2$. The objective is to minimize
$\norm{F}_{C(\fA)} = \|M_F\|_{B(H,H^{(n)})}$ over all functions $F$
such that $F(x_i)^*v_i = \overline{w_i}$.  Given the finite set
$\{x_1, \dots, x_n\} \subset X$, we denote $\fJ = \fJ_E = \{ G \in
C(\fA) : G(x_i)^*v_i = 0, \: i=1, \dots, n \}$. The column space
$C(\fA)$ is naturally a weak$^*$-closed $\fA$-bimodule and $\fJ$ is
weak$^*$-closed submodule of $C(\fA)$.

An application of standard duality arguments shows us that this
optimization problem is equivalent to computing the distance of $F$
from the submodule $\fJ$. One is usually interested in relating this
distance to the norm of the compression of the operator $M_f$ to a
semi-invariant subspace. In many cases a family of semi-invariant
subspaces is required.  We first establish an easy distance estimate
that always holds.

\begin{lem} \label{L:easyestimate} Suppose $H$ is a RKHS on a set $X$
  and that $\fA$ is a dual algebra of multipliers on $H$. If
  $x_1,\ldots,x_n \in X$ and $v_1,\ldots,v_n\in \ell^2$, then the
  distance from $F \in C(\fA)$ to $\fJ = \{ G \in C(\fA) : G(x_i)^*v_i
  = 0, \: i=1, \dots, n \}$ has the lower bound
\[ d(F, \fJ) \geq \sup_{L \in \Lat (\fA)}\|P_L M_F^*|_{\ol{C(\fA)L} \ominus \fJ L} \|.\] 
\end{lem}
\begin{proof}
For each $L \in \Lat(\fA)$, $F \in
C(\fA)$ and $G \in \fJ$ we have 
\begin{align*}
\| F - G \|_{C(\fA)} & \geq  \|P_{\ol{C(\fA)L} \ominus \fJ L} (M_F- M_G) P_L \|  \\
& =  \|P_{\ol{C(\fA)L} \ominus \fJ L} M_F P_L \|   \\
& =  \|P_L M_F^*|_{\ol{C(\fA)L} \ominus \fJ L} \| \qedhere.
\end{align*}
\endproof
\end{proof}

Given a subspace $L\in\Lat(\fA)$ we let $K_L = L\otimes \ell^2\ominus \ol{\fJ L}\subseteq L\otimes \ell^2$. Under certain assumptions on separation of points it is the case that $K_L$ is the span of the functions $\{k^L_{x_1}\otimes v_1, \ldots, k^L_{x_n}\otimes v_n\}$.  We denote this span $M_L$.  A simple computation shows that the operator $M_F^*|_{M_L}$ is a contraction if and only if the $n\times n$ matrix
\[ [(\inp{v_j}{v_i}-w_i\overline{w_j})K^L(x_i,x_j)] \]
is positive semidefinite. 

\begin{defn} \label{D:tangfam}
We will call a family of subspaces $\fL\subset \Lat(\fA)$ a \textit{tangential NP family} if for every choice of points $x_1,\ldots,x_n\in X$, $w_1,\ldots,w_n\in \bC$ and $v_1,\ldots,v_n\in \ell^2$, we have $d(F,\fJ) = \sup_{L\in\fL}\norm{P_LM_F^*|_{M_L}}$. 
\end{defn}
Given an algebra $\fA$ it is not necessarily true that $\Lat (\fA)$ is a tangential NP family (see Section 6 of \cite{DH}).

Suppose that $M\subseteq B(H)$ is a weak$^*$-closed subspace. We say
that $M$ has property $A_1(1)$ if given a weak$^*$-continuous linear
functional $\phi$ on $M$ with $\norm{\phi}<1$, there exist vectors
$x,y\in H$, with $\norm{x}\norm{y}<1$, such that $\phi(A) =
\inp{Ax}{y}$ for all $A\in M$ (see ~\cite{BFP85} for a detailed
treatment of predual factorization properties).  It is easy to check,
from the duality between trace-class operators and $B(H)$, that the
space $M$ viewed as $M\otimes I\subseteq B(H\otimes \ell^2)$ has
$A_1(1)$. This observation is essentially the starting point for the
distance formulae in Arveson~\cite{Arveson},
McCullough~\cite{McCull96}, and Raghupathi-Wick~\cite{RagWick}.  If
$H$ and $K$ are distinct Hilbert spaces, the weak$^*$-topology on
$B(H,K)$ is the one inherited from identifying $B(H,K)$ with the
natural subspace of $B(H \oplus K)$.

We will show that if the column space $C(\fA)$ has property $A_1(1)$,
then $\CycLat(\fA)$ is a tangential NP-family. In view of the comments
in the previous paragraph, this gives rise to a family of matrix
positivity conditions that are equivalent to the solvability of the
tangential interpolation problem.  We now state the required distance
formula; a proof of this will be given in the next section.

\begin{thm}\label{T:main}
Suppose $\fA$ is a weak-$*$-closed algebra of multipliers on a reproducing kernel Hilbert space $H$ and that the column space $\mathcal{C} (\fA)$, regarded as a weak-$*$-closed subspace of $B(H , H\otimes \ell^2)$, has property $A_1(1)$. Then the distance to the subspace $\fJ$ is given by
\begin{align*}
\dist (F, \fJ) = \sup_{L \in \CycLat(\fA)} \| P_L M_F^* |_{\K_L} \| .
\end{align*}
Suppose further that a functional $\phi$ in the predual $\C ( \fA)_{*}$ can be factored as $\phi (F) = \langle F h, K \rangle$ where $h$ does not vanish on at any  of the points $x_i$.  Then the distance formula above can be improved to
\begin{align*}
\dist (F, \fJ) = \sup \{\| P_L M_F^* |_{\M_L} \| : L = \fA[h] \in \CycLat {\fA}, \: h(x_i) \neq 0\}.
\end{align*}
Equivalently, given points $x_1,\ldots,x_n\in X$, $w_1,\ldots,w_n\in \bC$ and $v_1,\ldots,v_n\in\ell^2$, there exists a function $F\in C(\fA)$ such that $\norm{M_F}\leq 1$ and $F(x_i)^*v_i = \overline{w_i}$ if and only if 
$[(\inp{v_j}{v_i} - w_i\overline{w_j})K^L(x_i,x_j)]\geq 0$ for all $L = \fA[h]$, $h(x_i)\not=0$. 
\end{thm}

In Davidson-Hamilton \cite{DH}, it was shown that if a dual algebra of multipliers has property $A_1(1)$, then one may deduce a NP type theorem.  This is precisely the case of Theorem~\ref{T:main} when considering columns of length one.  Algebras of multipliers appear to be a natural setting for predual factorization results; a trend pioneered by Sarason \cite{Sar67}.

\section{Predual factorization and a distance formula}%
\label{S:predual}
It is entirely possible when dealing with subalgebras that there may
be no functions that satisfy the condition $F(x_i)^*v_i =
\overline{w_i}$. Therefore, the final claim in Theorem~\ref{T:main}
depends on the existence of at least one function $F\in C(\fA)$ such
that $F(x_i)^*v_i = \overline{w_i}$ (here there is no norm constraint
on $F$). The proof of the existence of such a function depends in a
crucial way on the fact that there is a function $h \in H$ that does
not vanish at the points $x_1,\ldots, x_n$.

The idea is fairly straightforward, although the computations are a
little involved. The algebra $\fA$ induces an equivalence relation on
the set $\{x_1,\ldots,x_n\}$ as follows: $x_i\sim x_j$ if and only if
$f(x_i) = f(x_j)$ for every $f \in \fA$. If $h$ is any function in $H$
that does not vanish on the $x_i$, it is not difficult to establish
that $x_i \sim x_j$ if and only if $k^L_{x_i}$ and $k_{x_j}^L$ are
linearly dependent, where $L = \fA[h]$. Let $X_1,\ldots,X_p$ denote
the equivalences classes. The equivalence relation leads naturally to
a partition of unity $\{e_k\}\subseteq \fA$ such that $\sum_{k=1}^p
e_k = 1$ and $e_k|_{X_k} = 1$. For example, to construct $e_1$, we
first choose $f_2,\ldots,f_p$ such that $f_i(x) = 1$ for $x\in X_1$
and $f_i(y) = 0$ for $y \in X_i$. Now let $e_1 = f_2\cdots f_p$. We
now follow the argument contained in the results of~\cite{RagWick}
which provides a construction of the function $F$ in terms of this
partition of unity. We state the result here for completeness.

\begin{prop}\label{P:arbitraryinterpolant}
  Suppose $\fA$ is a dual algebra of multipliers on $H$ and that
  $[(\inp{v_j}{v_i}-w_i\overline{w_j})K^L(x_i,x_j)]\geq 0$ for at
  least one subspace $L$ of the form $L = \fA[h]$, where
  $h(x_i)\not=0$. Then, there is a function $F \in C(\fA)$ such that
  $F(x_i)^*v_i = w_i$ for each $i$.
\end{prop}

Note that Proposition~\ref{P:arbitraryinterpolant} is immediate if
$\fA$ separates points in $X$. As mentioned earlier we can explicitly
write down a basis for the space $C(\fA)[h]\ominus \fJ[h]$ under the
assumption that the function $h$ does not vanish at any of the points
$x_1,\ldots,x_n$. As before, let $K_L = C(\fA)[h] \ominus \fJ[h]$ and
let $M_L = \spn \{ k_{x_i}^L \otimes v_i, \: i=1\dots n \}$.  It is
always the case that $M_L\subset K_L.$

\begin{lem}\label{L:kernspan}
  Suppose $\fA$ is a dual algebra of multipliers on $H$ and write $L =
  \fA[h]$ for $h \in H$. If $h$ does not vanish on any of the
  $x_i$ (so that $k_{x_i}^L \neq 0$), then $M_L = K_L$.
\end{lem}
\begin{proof}
  For the non-trivial inequality $K_L \subset M_L$, we use a
  dimension argument.  Let $\phi_i : C(\fA) \rightarrow \bC$ denote
  the functional $F \mapsto \langle F(x_i),v_i \rangle$ and note
  that $\fJ = \bigcap_{i=1}^n \ker ( \phi_i)$ has codimension at most
  $n$.  Thus $K_L$ is at most $n$-dimensional. Since $h$ does not
  vanish on the $x_i$, the vectors $\{P_Lk_{x_i}\}$ are linearly
  independent if the algebra $\fA$ separates the $x_i$.  In this
  case, $M_L$ is spanned by $n$ linearly independent vectors and we
  are done.

  If $\fA$ does not separate the $x_i$, then we can partition the set
  $\{x_1,\ldots,x_n\}$ into equivalence classes $X_1, \dots, X_p$,
  where $\fA$ identifies points in every $X_j$.  Let $\fJ_j$ denote the set of multipliers $F\in C(\fA)$ such that $F(x_i)^*v_i = 0$ for $x_i \in X_j$. It suffices to prove
  that
  \[ C(\fA)[h] \ominus \fJ_j [h] = \spn \{P_Lk_i \otimes v_i :
  x_i \in X_j \}\]%
  for each $j$.  We have $\fJ_{j} = \bigcap_{x_i \in X_j}
  \ker ( \phi_i)$, and since any $F$ in $C(\fA)$ only takes on a
  single value on $X_j$, the codimension of $\fJ_{j}$ is at most $m
  := \dim( \spn \{ v_i\})$.  On the other hand, since
  $P_Lk_{x_i} \neq 0$, the right hand side always has dimension
  at least $m$.  Since the right hand side is contained in the left,
  the proof is complete.
\end{proof}

We are now in a position to prove the main factorization theorem.

\begin{thm} \label{T:A1distance}
Suppose $\fA$ is a weak-$*$-closed algebra of multipliers on a reproducing kernel Hilbert space $H$ and that the column space $C (\fA)$, regarded as a weak-$*$-closed subspace of $B(H ,  H\otimes \ell^2)$, has property $A_1(1)$. Then the distance to the subspace $\fJ$ is given by
\begin{align*}
\dist (F, \fJ) = \sup_{L \in \CycLat(\fA)} \| P_L M_F^* |_{K_L} \| .
\end{align*}
Suppose further that a functional $\phi$ in the predual $\C ( \fA)_{*}$ can be factored as $\phi (F) = \langle F h, K \rangle$ where $h$ does not vanish on any of the $x_i$.  Then the distance formula above can be improved to
\begin{align*}
\dist (F, \fJ) = \sup \{\| P_L M_F^* |_{M_L} \| : L = \fA[h] \in \CycLat {\fA}, \: h(x_i) \neq 0\}.
\end{align*}
\end{thm}
\begin{proof}
  There is a contractive weak-$*$-continuous linear functional $\phi$
  on $C(\fA)$ such that $\dist (F, \fJ) = | \phi(F) |$.  Fix $\epsilon
  > 0$ and find vectors $h$ in $H$ and $K$ in $ H\otimes \ell^2$ with
  $\| h \| \| K\| < (1 + \epsilon)$ such that
\begin{align*}
\phi(F) =  \left \langle M_Fh, K \right \rangle. 
\end{align*}
Let $L = \fA [h]$ and replace $K$ with $(P_L \otimes I)K$.  Since $\phi(\fJ) = 0$, we see that $\fJ[h]$ is orthogonal to $K$.  It follows that $K \in K_L$ and so we have
\begin{align*}
\dist (F, \fJ) = |\langle M_Fh, K \rangle| =  |\langle P_{K_L}M_Fh, K \rangle| < \|P_LM_F^*|{K_L} \|(1 + \epsilon).
\end{align*}
It follows that $\dist (F, \fJ) \leq |P_LM_F^*|{K_L} \|$. The reverse inequality was already shown.
If, additionally, the function $h$ does not vanish on any of the $x_i$, then we may restrict the above supremum to those functions.  In this case Lemma~\ref{L:kernspan} implies $K_L = M_L$, and so the second statement follows.
\end{proof}

\section{Drury-Arveson space and complete NP kernels}%
\label{S:ds}
In this section we establish the fact that the multiplier algebra (and all of its unital, weakly closed subalgebras) of a complete NP kernel satisfies the assumptions of Theorem~\ref{T:main}. 

\emph{Drury-Arveson space $H^2_d$} is a reproducing kernel Hilbert space on the complex ball $\bB_d$
of $\bC^d$ (including $d=\infty$) with kernel
\[ k^d (w,z) := \dfrac1{1-\ip{w,z}} .\] 
The multiplier algebra $M(H^2_d)$ is generated by the coordinate functions $M_{z_1}, \dots, M_{z_d}$.  The row contraction $[M_{z_1}, \dots, M_{z_d}]$ serves as the model for contractive rows of commuting operators \cite{Arv}.

A reproducing kernel $k$ for $H$ is said to be \emph{complete} if the
matrix-valued NP theorem holds for $M(H)$.  Examples of such spaces
include Hardy and Dirichlet space on the disk, as well as the
Sobolev-Besov spaces on $\bB_d$. 
The kernel $k$ is said to be \emph{irreducible} if for distinct $x$ and $y$ in $X$, the functions $k_x$ and $k_y$ are linearly independent and  $\langle k_x, k_y \rangle \neq 0$.
It is well known that $k^d$ is a
irreducible and complete NP kernel (see \cite{AMc00} for a detailed treatment of complete NP kernels).

In fact, the Drury-Arveson kernel is universal among all complete NP
kernels as shown by McCullough\cite{McCull92,McCull94} and Quiggin
\cite{Quig93,Quig94}.  Another proof was provided by Agler and
McCarthy \cite{AMc00}.  Up to rescaling of kernels, they showed that
for any complete NP space $H$, there is a subset $S$ of $\bB_d$ such
that
\[ H = \overline{\spn}\{k^d_x : x \in S\}.  \]

Spans of kernel functions are always co-invariant for multiplication
operators. Consequently, every irreducible and complete NP space
corresponds to a co-invariant subspace of $M(H^2_d)$. It was shown in
\cite{DP98b} that the multiplier algebras of complete and irreducible
spaces are all complete quotients of the \emph{non-commutative
  analytic Toeplitz algebra} $\fL_d$.  This is the weak operator
closed algebra generated by the left regular representation of the
free semigroup $\bF_d^*$ on the full Fock space $H_n :=
\ell^2(\bF_d^*)$.  It follows immediately that they are complete
quotients of $M(H^2_d)$, which is shown directly by Arias and Popescu
\cite{AP00}.

We briefly describe the connection between invariant subspaces and
complete quotients of $\fL_d$ here.  If $\I$ is a \wot-closed,
two-sided ideal of $\fL_d$ with range $M=\ol{\I H_n}$, then
\cite{DP98b} shows that there is a normal, completely isometrically
isomorphic map from $\fL_d/ \I$ to the compression of $\fL_d$ to
$M^\perp$.  Conversely, if $M$ is an invariant subspace of both
$\fL_d$ and its commutant, the right regular representation algebra
$\fR_d$, then $\I = \{ A \in \fL_d : \ran A \subset M\}$ is a
\wot-closed ideal with range $M$.  In particular, if $\C$ is the
commutator ideal, it is shown that $M(H^2_d) \simeq \fL_d/\C$.
Moreover, the compression of both $\fL_d$ and $\fR_d$ to $H^2_d$ agree
with $M(H^2_d)$. On the other hand, if $N$ is a coinvariant subspace
of $H^2_d$, then $M=H_n \ominus N$ is invariant for both $\fL_d$ and
$\fR_d$.

For each $x \in H_n$, we may write $x = Ru$ where $R \in \fR_d$ is an
isometry, and $u$ is cyclic for $\fL_d$ (inner-outer factorization for
Fock space \cite{DP99}).  A deep result of Bercovici \cite{Berco98}
shows that an algebra of operators has the property $A_1(1)$ (in fact,
a much stronger property called $X_{0,1}$) if its commutant contains
two isometries with pairwise orthogonal ranges.  Consequently, $\fL_d$
and $\fL_d \otimes B(\ell^2)$ both have property $A_1(1)$ (when $d
\geq 2$).  If $d' > d$, there is the canonical embedding of $\bB_d$
into $\bB_{d'}$, and so there is no loss in assuming that $d \geq 2$.
It is essential to use this embedding for Hardy space, for example,
since $H^\infty \otimes B(\ell^2)$ does not have $A_1(1)$.

\begin{thm}\label{T:mainfactorization}
  Suppose $\I$ is a weak$^*$-closed ideal in $\fL_d$ and let $M =
  (\ran {\I})^\perp$. Let $\fA$ denote the quotient algebra $\fL_n /
  \I$ and form the column space $C(\fA) $.  Then $C(\fA)$, regarded as a
  weak$^*$-closed subspace of $B(M, M\otimes \ell^2)$ has property
  $A_1(1)$.  Moreover, for any $\epsilon > 0$ a weak$^*$-functional
  $\phi$ may be factored as $\phi(A) = \langle Au, V\rangle$ where $u$
  is cyclic for $\fA$ and $\|u\|\|v\| < 1 + \epsilon$.
\end{thm}
\begin{proof}
Suppose $\phi \in C(\fA)_*$ is of norm at most $1$ and let $Q : C(\fL_d) \rightarrow C(\fA)$ be given by $Q((B_i)) = (q(B_i))$, where $q: \fL_d \rightarrow \fA$ is the canonical quotient map.  Then $\phi \circ Q$ is a weak$^*$-functional on $C(\fL_d)$, which is a weak$^*$-closed subspace of $\fL_d \otimes B( \ell^2)$.  Property $A_1(1)$ is hereditary for weak$^*$-closed subspaces,
 and so for any $\varepsilon > 0$ there are vectors $x \in H_n$ and $Y = (y_i) \in H_n \otimes \ell^2$ so that $\phi \circ Q ([A_i]) = \langle (A_i)x, Y \rangle$ and $\| x \| \| Y \| < 1 + \varepsilon$.

As in the above discussion, let $R \in \fR_d$ and $u$ a cyclic vector for $\fL_d$ so that $x = Ru$.  Let $V = (R^* \otimes I)Y$ and observe that $\langle Au, V \rangle=\langle Ax, Y\rangle$ for any $A \in C(\fL_d)$.
 We also have $C(\I) u\perp V$ and 
 \[\ol{C(\I) u} = \ol{C(\I \fL_d) u} =  \ol{\I H_n} \otimes \ell^2 = M^\perp \otimes \ell^2.\] 
 It follows that $V \in M \otimes \ell^2$, which is a co-invariant subspace of $\fL_d \otimes B(\ell^2)$.

For $A \in C(\fA)$, find a $B \in C(\fL_d)$ be such that $Q(B) = A$.  We have
\begin{align*}
\phi(A) = \phi \circ Q(B) &= \langle Bu, V \rangle  \\
&= \langle Bu, (P_M \otimes I)V \rangle = \langle (P_M \otimes I)Bu, V \rangle  \\
&= \langle (P_M \otimes I)B(P_M \otimes I)u, V \rangle = \langle A(P_Mu), V \rangle.
\end{align*}
The property $A_1(1)$ now follows.  Note that the vector $u$ is cyclic for $\fL_d$.  Therefore the vector $P_Mu$ is cyclic for $\fA$ since 
\[\ol{\fA P_Mu} = \ol{P_M\fA P_Mu} = \ol{P_M \fL_dP_M u} = \ol{P_M \fL_d u} = M. \qedhere\] 
\end{proof}

The technique in the above proof was applied similarly in Theorem 5.2
of \cite{DH} and the idea is due to Arias-Popescu \cite{AP95}.
Applying the above theorem to the case where $M$ is the closed span of
kernel functions and using the classification of complete NP spaces,
we obtain the following corollary.

\begin{cor}\label{C:completePick}
  Suppose $H$ is a reproducing kernel Hilbert space with a 
   complete NP kernel.  Then $C(M(H))$ has property
  $A_1(1)$ and the additional property that any weak-$^*$-continuous
  functional on $C(M(H))$ can be factored as $\phi(A) = \langle Ah, K \rangle$
  where $h$ is cyclic for $M(H))$.
\end{cor}
\begin{proof}
When $k$ is an irreducible complete NP kernel, the result follows immediately from Theorem~\ref{T:mainfactorization}.  For an arbitrary complete Pick kernel on $X$, Lemma 7.2 of \cite{AMc00} says that we can write $X$ as the disjoint union of subsets $X_i$ with $k$ irreducible on each $X_i$ and $\langle k_x, k_y \rangle = 0$ precisely when $x$ and $y$ belong to different subsets.  Then define the mutually orthogonal subspaces $H_i = \ol{\spn}(k_\lambda: \lambda \in X_i)$ so that $H = \bigoplus_i H_i$.

For each multiplier $f \in M(H)$, the subspace $H_i$ is invariant for both $M_f$ and $M_f^*$. Thus$M(H) = \bigoplus_i M(H_i)$ and $M(H_i)$ is the multiplier algebra of an irreducible and complete  NP kernel.  It then follows from the irreducible case that a weak-$*$ functional on $M(H)$ has the desired factorization properties. \qedhere
\end{proof}
This result also extends to any dual algebra of multipliers on a
complete NP space.
\begin{cor}\label{C:completePick2}
  Suppose $\fA$ is a dual algebra of multipliers on a reproducing
  kernel Hilbert space $H$ with a complete NP kernel.
  Then $C(\fA)$ has property $A_1(1)$, and every weak-$*$-continuous
  functional can be factored as $\phi(A) = \langle Ah, K \rangle$ where $h$ is
  cyclic for the full multiplier algebra $M(H)$.
\end{cor}
\begin{proof}
  Property $A_1(1)$ is hereditary for dual subalgebras.  Moreover, any
  weak-$*$ functional on $\fA$ is the restriction of a weak-$*$
  functional on $M(H)$ with a small increase in norm.
\end{proof}

We can now apply Theorem~\ref{T:A1distance} to the above setting.

\begin{cor}\label{C:tangentialPick}
  Suppose $\fA$ is a dual algebra of multipliers on a complete NP
  space $H$.  Then
  \[ \L := \{k^L : L = \fA[h], \: h \in H \text{ cyclic for } M(H)
  \} \]%
  is a tangential Pick family for $\fA$.  Equivalently, the distance
  formula from $C(\fA)$ to $\fJ$
  \[ d(F, \fJ) = \sup \{\|(P_L M_F^*|_{M_L} \|: L \in \L \}\]%
  holds.
\end{cor}

\section{The Toeplitz corona problem for subalgebras}
\label{S:TC}
We will now apply the tangential interpolation results of the previous
section to obtain a Toeplitz corona theorem.  Suppose $\fA $ is a dual
algebra of multipliers on a reproducing kernel Hilbert space $H$.
Given functions $f_1,\dots,f_n$ in $\fA$, the Toeplitz corona problem
asks that if there is a $\delta > 0$ such that
\begin{align} \label{E:positive}
\sum_{i=1}^nM_{f_i}M_{f_i}^* \geq \delta^2 I,
\end{align}
is it possible to find functions $g_1,...,g_n$ in $\fA$ such that \[
f_1g_1 + \dots + f_ng_n = 1? \] In other words, the row operator
$[M_{f_1}, \dots, M_{f_n} ]$ does have a right inverse in
$B(H,H^{(n)})$, by Douglas's factorization theorem. The question is
whether there exists a right inverse with entries in $\fA$.  One
typically requires some type of norm control on the $g_i$ as well. By
considering the case where $n=1$, the constant $\delta^{-1}$ is easily
seen to be the optimal operator norm for the column $[M_{g_1}, \dots,
M_{g_n}]^T$.  Using our notation, given a multiplier $F \in
M(H^{(n)},H)$ with $M_FM_F^* \geq \delta^2I$, is there a multiplier $G
\in M(H,H^{(n)})$ such that $FG = 1$ and $\| G \|_{M(H,H^{(n)})} \leq
\delta^{-1}$?

For the algebra $H^\infty$ acting on Hardy space, this question was
answered affirmatively by Arveson in \cite{Arveson} using his famous
distance formula for nest algebras, albeit without optimal norm
control.  Under the hypothesis that each $f_i$ was contractive, he
showed that the functions $g_i$ could be chosen such that
$\|g_i\|_\infty \leq 4n\delta^{-3}$.  Schubert (\cite{Schubert})
obtained the result with optimal constants using the commutant lifting
theorem of Sz. Nagy and Foias.  This program was carried out in
substantial generality by Ball, Trent and Vinnikov \cite{BTV} for
multiplier algebras of complete NP spaces.

If the Toeplitz corona problem has an affirmative solution for the
full multiplier algebra $M(H)$, then for $f_1,\dots,f_n \in \fA$, the
condition $\sum_{i=1}^nM_{f_i}M_{f_i}^* \geq \delta^2 I$ certainly
implies the existence of the solutions $g_1,\dots,g_n $ in $M(H)$.
However, in order to require that these functions belong to the
subalgebra $\fA$, a stronger set of assumptions on the $f_i$ is
generally required. The following result, appearing as Proposition 4.2
in \cite{RagWick} says that if $\L$ is a tangential Pick family for
$\fA$ and that $M_F^L(M_F^L)^* \geq \delta^2I$ for every $L \in \L$,
then there are solutions $g_1,\dots,g_n$ in $\fA$.  We include the
proof for completeness.
\begin{prop}
Suppose $\fA$ is a dual algebra of multipliers on a reproducing kernel Hilbert space $H$ and that $\L$ is a tangential Pick family for $\fA$.  If $F \in R(\fA)$ and $M_F^L(M_F^L)^* \geq \delta^2I$ for each $L \in \L$, then there is a multplier $G \in C(\fA)$ with $\|M_G\| \leq \delta^{-1}$ such that $FG = 1$.
\end{prop}
\begin{proof}
  For any finite set of points $E = \{ x_1,\dots,x_k\}
  \subset X$, the positivity condition implies that
  \[ \left [ (\langle F(x_j)^*, F(x_i)^* \rangle -
    \delta^2 )k^L(x_i,x_j \right ] \geq 0 \]%
  for every $L \in \L$.  For $1 \leq i \leq k$, set $v_i = F(x_i)^*$
  and $w_i = \delta$ so that the above matrix is of the form in
  Theorem~\ref{T:main}.  Since $\L$ is a tangential Pick family, for
  each $E$ there is a corresponding contractive $G_E \in C(\fA)$ such
  that $\delta = \langle F(x_i), G_E(x_i)^* \rangle$ for $i =1 ,
  \dots, k$.  Since $\delta > 0$, we have
  $F(x_i)G_E(x_i) = \delta$.  Using a standard weak$^*$ approximation
  argument, there is a contractive multiplier $G \in R(\fA)$ such that
  $G|_E = G_E$ for every finite $E \subset X$.  It follows that $FG =
  \delta$ and that $\delta^{-1}G$ is the solution of required
  multiplier norm.
\end{proof}

In other words, in order to solve the
Toeplitz corona problem for an arbitrary subalgebra $\fA$, 
one requires that the row multiplier $M_F^L$
has a right inverse in $B(L,L^{(n)})$ for every $L \in \L$.
We can now use Corollary~\ref{C:tangentialPick} to solve the Toeplitz corona problem for subalgebras of $M(H)$, where $H$ is a complete NP space.  

\begin{thm} \label{T:completetopl}
Suppose $H$ is a RKHS a complete NP kernel and $\fA \subset M(H)$ is a dual algebra of multipliers on $H$.  If $F \in R(\fA)$ and $\delta > 0$ such that
\begin{align} \label{E:positive2}
M^L_F(M_F^L)^* \geq \delta^2I_L 
\end{align}
for every $L = \fA[h]$ where $h$ is a cyclic vector for $M(H)$, then there is a $G \in C(A)$ such that $FG =1$ and $\| M_G \| \leq \delta^{-1}$.
\end{thm}
Note that Theorem ~\ref{T:completetopl} works just as well for infinitely many $f_i$.
When $\fA = M(H)$ in Theorem~\ref{T:completetopl}, then $\fA[h]= H$ for any cyclic vector $h$.  Consequently, the scalar-valued version of the Ball-Trent-Vinnikov result \cite{BTV} is recovered as a special case.

\begin{eg}
We will demonstrate how the hypothesis in Theorem~\ref{T:completetopl} can be replaced by a slightly simpler one for `large' subalgebras of $M(H)$.  Suppose $H$ is a RKHS with irreducible and complete NP kernel and that $\I$ is a weak-$*$-closed ideal in $M(H)$ of \emph{finite codimension} $k$.  Form the subalgebra $\fA := \bC + \I = \{ c + g: c \in \bC, g \in \I \}$.

If $h$ is a cyclic vector for $M(H)$, let $M := \I[h] = \ol{\I H}$  We have 
\[\fA[h] = \ol{\spn (h, \I[h])} = \ol{\spn (h, M)} = P_M^\perp h \oplus M. \]
Since $\I$ is of finite codimension, so is $M$.  Find an orthogonal basis $\{e_1, \dots, e_p \}$ for $M^\perp$ where $p \leq k$.  For each cyclic vector $h$, there are scalars $a_1, \dots, a_p \in \bC$ such that $h = a_1e_1 + \dots + a_pe_p$.  Rescaling if necessary, we may assume that $|a_1|^2 + \dots + |a_p|^2 = 1$, i.e. $a := (a_1, \dots a_p) \in \partial \bB_p$. For $a \in \partial \bB_p$, let $L_a = \bC(a_1e_1 + \dots a_pe_p) \oplus M$ denote the invariant subspace associated to $a$ and set $M_F^a := M_f^{L_a}$.  The hypothesis in Theorem~\ref{T:completetopl} may be replaced by 
\[ M_F^a(M_F^a)^* \geq \delta^2I_{L_a}, \: a \in \partial \bB_p. \]
For instance, if $H = H^2_2$ and $\I = \langle z_1z_2 \rangle$ is the ideal generated by the monomial $z_1z_2$, one is required to check positivity with respect to the invariant subspaces
\[L_{a,b,c} = \bC ( a + bz_1 + cz_2) \oplus \spn (z_1^kz_2^l: k,l \geq 1) \text{ for } (a,b,c) \in \partial \bB_3 .\] 
\end{eg}
 
 \begin{rem}
It is natural to ask if one is required to check the inequality (\ref{E:positive2}) for \emph{every} such $L$.  In \cite{DPRS09}, analysis of this type was carried out for the NP problem on the algebra $\bC + z^2H^\infty$ acting on $H^2$. The answer is surprising: one is required to check essentially all invariant subspaces.  Consequently, all subspaces are required for the distance formula in Theorem~\ref{T:main}.
On the other hand, it may be possible to retrieve Theorem~\ref{T:completetopl} without the use of tangential interpolation.  It would be very interesting to have an alternative proof that requires fewer subspaces.
\end{rem}


\end{document}